\documentclass[draft]{amsproc}
\usepackage{amssymb}
\usepackage{graphicx}
\usepackage{euscript}
\usepackage{mathrsfs} 
\usepackage{units}   
\usepackage{color}

\makeatletter
\@namedef{subjclassname@2010}{%
\textup{2010} Mathematics Subject Classification}
\makeatother

\numberwithin{equation}{section}

\newtheorem{thm}{Theorem}[section]

\newtheorem{pro}[thm]{Proposition}

\newtheorem*{thm*}{Theorem}

\theoremstyle{remark}
\newtheorem{rem}[thm]{Remark}

\newtheorem*{opq}{Question}

\theoremstyle{definition}

\DeclareMathOperator{\dess}{{\mathsf{Des}}}
\DeclareMathOperator{\dzii}{{\mathsf{Chi}}}

\DeclareMathOperator{\koo}{{\mathsf{root}}}

\DeclareMathOperator{\paa}{{\mathsf{par}}}

\newcommand*{\cbb}{\mathbb C}

\newcommand*{\card}[1]{\mathrm{card}(#1)}
\newcommand*{\des}[1]{{\dess(#1)}}
\newcommand*{\dz}[1]{{\EuScript D}(#1)}
\newcommand*{\dzi}[1]{\dzii(#1)}
\newcommand*{\dzin}[2]{\dzii^{\langle#1\rangle}(#2)}

\newcommand*{\escr}{{\mathscr{E}_V}}

\newcommand*{\Ge}{\geqslant}
\newcommand*{\hh}{\mathcal H}

\newcommand*{\kk}{\mathcal K}

\newcommand*{\lambdab}{{\boldsymbol\lambda}}

\newcommand*{\Le}{\leqslant}

\newcommand*{\pa}[1]{\paa(#1)}

\newcommand*{\slam}{S_{\boldsymbol \lambda}}
\newcommand*{\smalloplus}{\raise0pt\hbox{$\scriptscriptstyle \oplus$}}

\newcommand*{\tcal}{{\mathscr T}}

\begin{document}
   \title[A hyponormal weighted shift whose
square has trivial domain] {A hyponormal weighted
shift on a directed tree whose square has trivial
domain}
   \author[Z.\ J.\ Jab{\l}o\'nski]{Zenon Jan Jab{\l}o\'nski}
\address{Instytut Matematyki, Uniwersytet Jagiello\'nski,
ul.\ \L ojasiewicza 6, PL-30348 Kra\-k\'ow, Poland}
   \email{Zenon.Jablonski@im.uj.edu.pl}
   \author[I.\ B.\ Jung]{Il Bong Jung}
   \address{Department of Mathematics, Kyungpook National
University, Daegu 702-701, Korea}
   \email{ibjung@knu.ac.kr}
   \author[J.\ Stochel]{Jan Stochel}
\address{Instytut Matematyki, Uniwersytet Jagiello\'nski,
ul.\ \L ojasiewicza 6, PL-30348 Kra\-k\'ow, Poland}
   \email{Jan.Stochel@im.uj.edu.pl}
   \thanks{Research of the first
and the third authors was supported by the MNiSzW
(Ministry of Science and Higher Education) grant NN201
546438 (2010-2013). The second author was supported by
the National Research Foundation of Korea (NRF) grant
funded by the Korea government (MEST) (No.\
R01-2008-000-20088-0).}
   \begin{abstract}
It is proved that, up to isomorphism, there are only
two directed trees that admit a hyponormal weighted
shift with nonzero weights whose square has trivial
domain. These are precisely those enumerable directed
trees, one with root, the other without, whose every
vertex has enumerable set of successors.
   \end{abstract}
   \maketitle
   \section{Introduction} In a recent paper
\cite{b-j-j-s} a question of subnormality of unbounded
weighted shifts on directed trees has been
investigated. A criterion for subnormality of such
operators whose $C^\infty$-vectors are dense in the
underlying Hilbert space has been established (cf.\
\cite[Theorem 5.2.1]{b-j-j-s}). It has been written in
terms of consistent systems of Borel probability
measures. The assumption that the operator in question
has a dense set of $C^\infty$-vectors diminishes the
class of weighted shifts on directed trees to which
this criterion can be applied (note that the set of
all $C^\infty$-vectors of a classical, unilateral or
bilateral, weighted shift is always dense in the
underlying Hilbert space). Unfortunately, there is no
general {\em criterion} for subnormality of densely
defined operators that have small set of
$C^\infty$-vectors. The known characterizations of
subnormality of unbounded Hilbert space operators
require the existence of additional objects (like
semispectral measures, elementary spectral measures or
sequences of unbounded operators) that have to satisfy
appropriate, more or less complicated, conditions
(cf.\ \cite{bis,foi,FHSz,sz4}). Among subnormal
operators having small set of $C^\infty$-vectors, the
symmetric ones (which are always subnormal, see
\cite[Theorem 1 in Appendix I.2]{a-g}) play an
essential role. According to \cite{nai} (see also
\cite{cher}) there are closed symmetric operators
whose squares have trivial domain. Unfortunately,
symmetric weighted shifts on directed trees are
automatically bounded; the same is true for formally
normal weighted shifts on directed trees (cf.\
\cite[Proposition 3.1]{j-j-s2}).

The above discussion leads to the following problem.
   \begin{opq}
Does there exist a subnormal weighted shift on a
directed tree with nonzero weights whose square has
trivial domain?
   \end{opq}
At present, this question is unanswered (the reason
for this is explained partially in the previous
paragraph). However, as is shown in Theorem
\ref{hypdz20}, there are injective hyponormal weighted
shifts on directed trees with nonzero weights whose
squares have trivial domain. What is more, it is
proved in Theorem \ref{sqr20} that the only directed
trees admitting densely defined weighted shifts with
nonzero weights whose squares have trivial domain are
those enumerable directed trees whose every vertex has
enumerable set of successors (children).
   \section{Preliminaries} In what
follows, $\cbb$ stands for the set of all complex
numbers. Let $A$ be an operator in a complex Hilbert
space $\hh$ (all operators considered in this paper
are linear). Denote by $\dz{A}$ and $A^*$ the domain
and the adjoint of $A$ (in case it exists). A closed
densely defined operator $N$ in $\hh$ is called {\em
normal} if $N^*N=NN^*$. A densely defined operator $S$
in $\hh$ is said to be {\em subnormal} if there exists
a complex Hilbert space $\kk$ and a normal operator
$N$ in $\kk$ such that $\hh \subseteq \kk$ (isometric
embedding) and $Sh = Nh$ for all $h \in \dz S$.
Finally, a densely defined operator $S$ in $\hh$ is
called {\em hyponormal} if $\dz{S} \subseteq \dz{S^*}$
and $\|S^*f\| \Le \|Sf\|$ for all $f \in \dz S$. It is
well-known that subnormal operators are hyponormal
(but not conversely) and that hyponormal operators are
closable and their closures are hyponormal (subnormal
operators have an analogous property). We refer the
reader to \cite{b-s,weid} for basic facts on unbounded
operators, \cite{con2,s-sz1,s-sz2,s-sz3,StSz2} for the
foundations of the theory of (bounded and unbounded)
subnormal operators and \cite{ot-sch,jj1,jj2,jj3,sto}
for elements of the theory of unbounded hyponormal
operators.

Let $\tcal=(V,E)$ be a directed tree ($V$ and $E$
always stand for the sets of vertices and edges of
$\tcal$, respectively). If $\tcal$ has a root, which
will always be denoted by $\koo$, then we write
$V^\circ:=V\setminus \{\koo\}$; otherwise, we put
$V^\circ = V$. Set $\dzi u = \{v\in V\colon (u,v)\in
E\}$ for $u \in V$. If for a given vertex $u \in V$
there exists a unique vertex $v\in V$ such that
$(v,u)\in E$, then we denote it by $\pa u$. The
correspondence $u \mapsto \pa u$ is a partial function
from $V$ to $V$. For an integer $n \Ge 1$, the
$n$-fold composition of the partial function $\paa$
with itself will be denoted by $\paa^n$. Let $\paa^0$
stand for the identity map on $V$. We call $\tcal$
{\em leafless} if $V =\{u \in V \colon \dzi u \neq
\varnothing\}$. If $W \subseteq V$, we put $\dzi W =
\bigcup_{v \in W} \dzi v$ and $\des W =
\bigcup_{n=0}^\infty \dzin n W$, where $\dzin{0}{W} =
W$ and $\dzin{n+1}{W} = \dzi{\dzin{n}{W}}$ for all
integers $n\Ge 0$. For $u \in V$, we set $\dzin n
u=\dzin n {\{u\}}$ and $\des{u}=\des{\{u\}}$.
Combining equalities (2.1.3), (6.1.3) and (2.1.10) of
\cite{j-j-s} with \cite[Corollary 2.1.5]{j-j-s}, we
obtain
   \allowdisplaybreaks
   \begin{align}   \label{roz}
V^\circ & = \bigsqcup_{u\in V} \dzi u,
   \\
\label{num1} \dzin{n+1}{u} & = \bigsqcup_{v \in
\dzin{n}{u}} \dzi{v}, \quad u \in V, \, n =0,1,2,
\ldots,
   \\
\des u & = \bigsqcup_{n=0}^\infty \dzin n u, \quad
u\in V, \label{num3}
   \\
\des{u_1} \cap \des{u_2} & = \varnothing, \quad u_1,
u_2 \in \dzi{u},\, u_1 \neq u_2,\, u \in V,
\label{num3+}
   \\
V &= \des{\koo} \quad \text{provided that $\tcal$ has
a root}, \label{korzen}
   \end{align}
where the symbol $\bigsqcup$ is reserved to denote
pairwise disjoint union of sets.

Let $\ell^2(V)$ be the Hilbert space of all square
summable complex functions on $V$ equipped with the
standard inner product. For $u \in V$, we define $e_u
\in \ell^2(V)$ to be the characteristic function of
the one point set $\{u\}$. The family $\{e_u\}_{u\in
V}$ is an orthonormal basis of $\ell^2(V)$. Denote by
$\escr$ the linear span of $\{e_u\colon u \in V\}$.
Given $\lambdab = \{\lambda_v\}_{v \in V^\circ}
\subseteq \cbb$, we define the operator $\slam$ in
$\ell^2(V)$ by
   \begin{align*}
   \begin{aligned}
\dz {\slam} & = \{f \in \ell^2(V) \colon
\varLambda_\tcal f \in \ell^2(V)\},
   \\
\slam f & = \varLambda_\tcal f, \quad f \in \dz
{\slam},
   \end{aligned}
   \end{align*}
where $\varLambda_\tcal$ is the map defined on
functions $f\colon V \to \cbb$ via
   \begin{align} \label{lamtauf}
(\varLambda_\tcal f) (v) =
   \begin{cases}
\lambda_v \cdot f\big(\pa v\big) & \text{ if } v\in
V^\circ,
   \\
0 & \text{ if } v=\koo.
   \end{cases}
   \end{align}
$\slam$ is called a {\em weighted shift} on the
directed tree $\tcal$ with weights $\{\lambda_v\}_{v
\in V^\circ}$. Note that any weighted shift $\slam$ on
$\tcal$ is a closed operator (cf.\ \cite[Proposition
3.1.2]{j-j-s}). Combining Propositions 3.1.3, 3.1.7
and 3.1.10 of \cite{j-j-s}, we get the following fact
(hereafter we adopt the convention that
$\sum_{v\in\varnothing} x_v=0$).
   \begin{pro}\label{bas}
Let $\slam$ be a weighted shift on a directed tree
$\tcal$ with weights $\lambdab = \{\lambda_v\}_{v \in
V^\circ}$. Then the following assertions hold{\em :}
   \begin{enumerate}
   \item[(i)] $e_u$ is in $\dz{\slam}$ if and only if
$\sum_{v\in\dzi u} |\lambda_v|^2 < \infty$; if $e_u
\in \dz{\slam}$, then
   \begin{align} \label{eu}
\slam e_u = \sum_{v\in\dzi u} \lambda_v e_v \quad
\text{and} \quad \|\slam e_u\|^2 = \sum_{v\in\dzi u}
|\lambda_v|^2,
   \end{align}
   \item[(ii)] $\slam$ is densely defined if and only
if $\escr \subseteq \dz{\slam}$,
   \item[(iii)] $\slam$ is injective if and only if
$\tcal$ is leafless and $\sum_{v\in\dzi u}
|\lambda_v|^2 > 0$ for every $u\in V$,
   \item[(iv)] if $\overline{\dz{\slam}}=\ell^2(V)$
and $\lambda_v \neq 0$ for all $v \in V^\circ$, then
$V$ is at most countable.
   \end{enumerate}
   \end{pro}
   \section{Directed trees admitting $\slam$'s with
$\dz{\slam^2}=\{0\}$} The proof of Theorem \ref{sqr20}
below contains a method of constructing densely
defined weighted shifts $\slam$ on directed trees with
nonzero weights such that $\dz{\slam^2}=\{0\}$. By
imposing carefully tailored restrictions on weights,
we will show in Theorem \ref{hypdz20} below how to use
this method to construct hyponormal weighted shifts on
directed trees with the aforesaid properties.
   \begin{thm}\label{sqr20}
Let $\tcal$ be a directed tree. Then the following
assertions are equivalent{\em :}
   \begin{enumerate}
   \item[(i)] there exists a family
$\lambdab=\{\lambda_v\}_{v\in V^\circ}$ of nonzero
complex numbers such that
$\overline{\dz{\slam}}=\ell^2(V)$ and
$\dz{\slam^2}=\{0\}$,
   \item[(ii)] $\card{\dzi{u}}=\aleph_0$ for every
$u \in V$.
   \end{enumerate}
Moreover, if $\slam$ is as in {\em (i)}, then $\slam$
is injective.
   \end{thm}
   \begin{proof} Fix $\lambdab=\{\lambda_v\}_{v\in V^\circ}
\subseteq \cbb$. We show that
   \begin{enumerate}
   \item[($\dag$)]  a complex function $f$ on $V$
belongs to $\dz{\slam^2}$ if and only
if\footnote{\;with the convention that $0 \cdot \infty
= 0$}
   \begin{align} \label{dof2}
\sum_{u \in V} \Big(1+\zeta_u^2 + \sum_{v \in \dzi{u}}
\zeta_v^2 |\lambda_v|^2\Big) |f(u)|^2 < \infty,
   \end{align}
   \end{enumerate}
where $\zeta_u:=\sqrt{\sum_{v \in \dzi{u}} |\lambda_v|^2}$
for $u \in V$.

Indeed, by \cite[Proposition 3.1.3]{j-j-s}, a complex
function $f$ on $V$ belongs to $\dz{\slam}$ if and
only if $f \in \ell^2(V)$ and $\sum_{u \in V}
\zeta_u^2 |f(u)|^2 < \infty$. Hence a complex function
$f$ on $V$ belongs to $\dz{\slam^2}$ if and only if
$\sum_{u \in V} (1+\zeta_u^2) |f(u)|^2 < \infty$ and
$\sum_{u\in V} \zeta_u^2 \, |(\slam f)(u)|^2 <
\infty$. Since the following equalities hold for $f
\in \dz{\slam}$,
   \begin{align*}
\sum_{u\in V} \zeta_u^2 \, |(\slam f)(u)|^2 &
\overset{\eqref{lamtauf}}= \sum_{u\in V^\circ}
\zeta_u^2 \, |\lambda_u|^2 |f(\paa(u))|^2
   \\
& \overset{\eqref{roz}}= \sum_{u\in V} \sum_{v\in
\dzi{u}} \zeta_v^2 |\lambda_v|^2 |f(\paa(v)|^2
   \\
&\hspace{1.8ex} = \sum_{u\in V} \Big(\sum_{v\in
\dzi{u}} \zeta_v^2 |\lambda_v|^2\Big) |f(u)|^2,
   \end{align*}
we see that a complex function $f$ on $V$ belongs to
$\dz{\slam^2}$ if and only if \eqref{dof2} holds.

   (i)$\Rightarrow$(ii) Let $\slam$ be as in (i). By
Proposition \ref{bas}(iv), $V$ is countable. Thus each
$\dzi{u}$ is countable. Suppose that, contrary to our
claim, (ii) does not hold. Then there exists $u_0 \in
V$ such that $\dzi{u_0}$ is finite. Since $\slam$ is
densely defined, we infer from assertions (i) and (ii)
of Proposition \ref{bas} that $\zeta_v < \infty$ for
all $v \in V$. Hence
   \begin{align*}
1+\zeta_{u_0}^2 + \sum_{v \in \dzi{u_0}} \zeta_v^2
|\lambda_v|^2 < \infty.
   \end{align*}
This, combined with ($\dag$), implies that $f=e_{u_0} \in
\dz{\slam^2}$, which contradicts (i).

   (ii)$\Rightarrow$(i) First, we show that
   \begin{enumerate}
   \item[($\ddag$)] for each $(\vartheta,u) \in (0,\infty)
\times V$ there exists $\{\lambda_v\}_{v \in \des{u}}
\subseteq (0,\infty)$ such that
   \begin{align}       \label{kwkw1}
\lambda_u^2 & =\vartheta,
      \\
\Big(\sum_{w\in \dzi{v}} \lambda_w^2\Big) \lambda_v^2
& = 1, \quad v\in \des{u}. \label{kwkw}
   \end{align}
   \end{enumerate}

To do so, we fix $u \in V$ and set $X_n = \dzin{0}{u}
\sqcup \cdots \sqcup \dzin{n}{u}$ for $n\Ge 1$, and
$X_0= \dzin{0}{u}$. Since, by \eqref{num3}, $\des{u} =
\bigcup_{n=1}^\infty X_n$, we can construct the
required family inductively. For $n=1$, we put
$\lambda_u=\sqrt{\vartheta}$ and choose a family
$\{\lambda_v\}_{v \in \dzi{u}}\subseteq (0,\infty)$
such that $\big(\sum_{v\in \dzi{u}} \lambda_v^2\big)
\vartheta = 1$ (this is possible because $\dzi{u}$ is
nonempty and countable). Fix $n\Ge 1$, and assume that
we already have a family $\{\lambda_v\}_{v \in
X_n}\subseteq (0,\infty)$ such that
$\lambda_u^2=\vartheta$ and $\big(\sum_{w\in \dzi{v}}
\lambda_w^2\big) \lambda_v^2=1$ for all $v \in
X_{n-1}$. Then for every $v \in \dzin{n}{u}$ we can
choose a family $\{\lambda_w\}_{w \in \dzi{v}}
\subseteq (0,\infty)$ such that $\big(\sum_{w\in
\dzi{v}} \lambda_w^2\big)\lambda_v^2=1$. In view of
\eqref{num1}, this gives us the family
$\{\lambda_v\}_{v \in \dzin{n+1}{u}}$ such that
$\big(\sum_{w\in \dzi{v}} \lambda_w^2\big)
\lambda_v^2=1$ for all $v \in X_{n}$. Now by induction
we are done.

If $\tcal$ has a root, then combining ($\dag$) and
($\ddag$) with \eqref{korzen} and Proposition
\ref{bas}(i) does the job (the number $\lambda_{\koo}$
can be chosen arbitrarily).

Suppose now that $\tcal$ is rootless. Take $u_1 \in V$
and set $u_2 = \paa(u_1)$. By $(\ddag)$, there exists
a family $\{\lambda_v\}_{v \in \des{u_1}} \subseteq
(0,\infty)$ with $\lambda_{u_1}=\frac 1{\sqrt2}$,
which satisfies \eqref{kwkw} with $u_1$ in place of
$u$. In the next step we construct a new family
$\{\lambda_v\}_{v \in \des{u_2}\setminus \des{u_1}}
\subseteq (0,\infty)$ with
$\lambda_{u_2}=\frac{1}{\sqrt{2}}$ such that the
extended family $\{\lambda_v\}_{v \in \des{u_2}}$
satisfies \eqref{kwkw} with $u_2$ in place of $u$. For
this, note that
   \begin{align}   \label{des-des}
\des{u_2} \setminus \des{u_1} \overset{\eqref{num3+}}=
\{u_2\} \sqcup \bigsqcup_{u \in \dzi{u_2} \setminus
\{u_1\}} \des{u}.
   \end{align}
Set $\lambda_{u_2}= \frac 1{\sqrt{2}}$ and choose a
family $\{\vartheta_u\}_{u \in \dzi{u_2} \setminus
\{u_1\}} \subseteq (0,\infty)$ such that
   \begin{align}    \label{des21}
\Big(\sum_{u \in \dzi{u_2} \setminus \{u_1\}}
\vartheta_u + \lambda_{u_1}^2\Big) \lambda_{u_2}^2= 1.
   \end{align}
Applying ($\ddag$) to $u \in \dzi{u_2} \setminus
\{u_1\}$ and $\vartheta=\vartheta_u$, we get the
family $\{\lambda_v\}_{v \in \des{u}} \subseteq
(0,\infty)$ satisfying \eqref{kwkw1} and \eqref{kwkw}
with $\vartheta=\vartheta_u$. This, together with
\eqref{des21}, leads to $\big(\sum_{u\in \dzi{u_2}}
\lambda_u^2\big) \lambda_{u_2}^2 = 1$. In view of
\eqref{des-des}, our construction is complete.
Applying an induction argument (with
$\lambda_{u_n}=\frac{1}{\sqrt{2}}$ for $n\Ge 2$) and
using the fact that $V=\bigcup_{k=0}^\infty
\des{\paa^k(u_1)}$ (cf.\ \cite[Proposition
2.1.6]{j-j-s}), we construct a family
$\lambdab=\{\lambda_v\}_{v\in V} \subseteq (0,\infty)$
such that $\zeta_v^2 \, \lambda_v^2 = 1$ for all $v
\in V$. This, combined with ($\dag$) and Proposition
\ref{bas}(i), gives (i).

The ``moreover'' part follows from (ii) and
Proposition \ref{bas}(iii).
   \end{proof}
Our method enables us to construct $\slam$ with the
additional property that $\dz{\slam} \nsubseteq
\dz{\slam^*}$, which is opposite to what happens in
Theorem \ref{hypdz20} below.
   \begin{thm} \label{dziedz1}
If $\tcal$ is a directed tree such that
$\card{\dzi{u}}=\aleph_0$ for every $u \in V$, then
there exists a family $\lambdab=\{\lambda_v\}_{v\in
V^\circ}$ of nonzero complex numbers such that $\slam$
is injective and densely defined, $\dz{\slam}
\nsubseteq \dz{\slam^*}$ and $\dz{\slam^2}=\{0\}$.
   \end{thm}
   \begin{proof}
To achieve this, we proceed as in the proof of
implication (ii)$\Rightarrow$(i) of Theorem
\ref{sqr20} with one exception, namely, we strengthen
($\ddag$) by requiring, in addition to \eqref{kwkw1}
and \eqref{kwkw}, that
   \begin{align}  \label{numJ}
\sup_{v\in \dzi{u}}\sum_{w \in \dzi v} \frac
{\lambda_w^4} {1+ \lambda_w^2}= \infty.
   \end{align}
This in turn can be deduced from the following fact:
   \begin{align} \label{alph}
   \begin{minipage}{70ex}
for every real number $\alpha > 0$, there exists a
sequence $\{\lambda_n\}_{n=1}^\infty \subseteq
(0,\infty)$ such that $|\lambda_{1} - \alpha| < 1$ and
$\sum_{n=1}^\infty \lambda_n^2 = \alpha^2$.
   \end{minipage}
   \end{align}
Indeed, arguing as in the proof of ($\ddag$), we find
a family $\{\lambda_v\}_{v\in \dzi{u}} \subseteq
(0,\infty)$ such that $\big(\sum_{v\in \dzi{u}}
\lambda_v^2\big) \vartheta = 1$. Then evidently
$\sup_{v \in \dzi{u}} 1/\lambda_v^2=\infty$. In the
next step, using \eqref{alph}, we construct a family
$\{\lambda_w\}_{w\in \dzin{2}{u}}$ such that
$\big(\sum_{w\in \dzi{v}} \lambda_w^2\big) =
1/\lambda_v^2$ for every $v\in \dzi{u}$ and $\sup_{w
\in \dzin{2}{u}} \lambda_w^2 = \infty$. This, combined
with \eqref{num1}, implies \eqref{numJ}. The rest of
the proof goes through as for ($\ddag$), with hardly
any changes. It follows from \eqref{eu} and
\eqref{kwkw} that $\|\slam e_w\|^2= 1/\lambda_w^2$ for
all $w \in \des{u}$, which together with \eqref{numJ}
implies that $\sup_{v\in V}\sum_{w \in \dzi v} \frac
{|\lambda_w|^2} {1+ \|\slam e_w\|^2}= \infty$. By
applying \cite[Theorem 4.1.1]{j-j-s}, we deduce that
$\dz{\slam} \nsubseteq \dz{\slam^*}$. Obviously, such
$\slam$ is never hyponormal.
   \end{proof}
   \section{Hyponormal weighted shifts $\slam$ with
$\dz{\slam^2}=\{0\}$} Let us recall a characterization
of hyponormality of weighted shifts on directed trees
with nonzero weights (in view of \cite[Proposition
5.3.1]{b-j-j-s}, there is no loss of generality in
assuming that underlying directed trees are leafless).
    \begin{thm}[\mbox{\cite[Theorem
5.1.2 and Remark 5.1.5]{j-j-s}}] \label{hyp} Let
$\slam$ be a densely defined weighted shift on a
leafless directed tree $\tcal$ with nonzero weights
$\lambdab = \{\lambda_v\}_{v \in V^\circ}$. Then
$\slam$ is hyponormal if and only if
    \begin{gather*}
\sum_{v \in \dzi{u}} \frac {|\lambda_v|^2}{\|\slam
e_v\|^2} \Le 1, \quad u \in V.
    \end{gather*}
    \end{thm}
Now we show that there are hyponormal weighted shifts
$\slam$ with $\dz{\slam^2}=\{0\}$.
   \begin{thm} \label{hypdz20}
If $\tcal$ is a directed tree such that
$\card{\dzi{u}}=\aleph_0$ for every $u \in V$, then
there exists a family $\lambdab=\{\lambda_v\}_{v\in
V^\circ}$ of nonzero complex numbers such that $\slam$
is injective and hyponormal, and $\dz{\slam^2}=\{0\}$.
   \end{thm}
   \begin{proof}
We modify the proof of implication
(ii)$\Rightarrow$(i) of Theorem \ref{sqr20}. First we
note that for each positive real number $r$, there
exists a sequence $\{r_n\}_{n=1}^\infty \subseteq
(0,1)$ such that $(\sum_{j=1}^\infty r_j)\, r=1$ and
$\sum_{j=1}^\infty r_j^2 \Le 1$ (e.g.,
$r_j=\frac{1}{rn}$ for $1 \Le j \Le n-1$, and
$r_j=\frac{1}{rn2^{j-n+1}}$ for $j\Ge n$, where $n \Ge
2$ is chosen so that $\frac{1}{r^2 n} \Le 1$). This
fact, when incorporated to the proof of ($\ddag$),
leads to
   \begin{enumerate}
   \item[($\ddag\ddag$)] for each $(\vartheta,u)
\in (0,\infty) \times V$ there exists
$\{\lambda_v\}_{v \in \des{u}} \subseteq (0,1)$ such
that $\lambda_u^2 =\vartheta$, $\big(\sum_{w\in
\dzi{v}} \lambda_w^2\big) \lambda_v^2 = 1$ and
$\sum_{w\in \dzi{v}} \lambda_w^4 \Le 1$ for all $v\in
\des{u}$.
   \end{enumerate}
If $\tcal$ has a root, then applying
\mbox{($\ddag\ddag$)} to $u=\koo$ and $\vartheta=1$ we
get a family $\lambdab=\{\lambda_v\}_{v\in V^\circ}
\subseteq (0,1)$ such that
   \begin{align}     \label{V}
\text{$\big(\sum_{w\in \dzi{v}} \lambda_w^2\big)
\lambda_v^2 = 1$ and $\sum_{w\in \dzi{v}} \lambda_w^4
\Le 1$ for all $v\in V$.}
   \end{align}
Suppose now that $\tcal$ is rootless. It is easily
seen that for every $r\in(0,1)$, there exists a
sequence $\{r_j\}_{j=1}^\infty \subseteq (0,1)$ such
that $r + \sum_{j=1}^\infty r_j = 2$ and $r^2 +
\sum_{j=1}^\infty r_j^2 \Le 1$. This fact combined
with the proof of Theorem \ref{sqr20} (use
\mbox{($\ddag\ddag$)} in place of ($\ddag$)) enables
us to construct a family $\lambdab=\{\lambda_v\}_{v\in
V} \subseteq (0,1)$ that satisfies \eqref{V}.

Since $\card{\dzi{u}}=\aleph_0$ for all $u \in V$, we
infer from assertions (i) and (iii) of Proposition
\ref{bas}, \eqref{V} and ($\dag$) that $\slam$ is
injective and densely defined, and
$\dz{\slam^2}=\{0\}$. It follows from \eqref{eu} and
the equality in \eqref{V} that $\lambda_v^2 = \|\slam
e_v\|^{-2}$ for all $v\in V^\circ$, and thus
   \begin{align*}
\sum_{v \in \dzi{u}} \frac {\lambda_v^2}{\|\slam
e_v\|^2} = \sum_{v \in \dzi{u}} \lambda_v^4
\overset{\eqref{V}}\Le 1, \quad u \in V,
   \end{align*}
which in view of Theorem \ref{hyp} completes the
proof.
   \end{proof}
   \begin{rem} \label{dziedz2}
In view of Theorems \ref{dziedz1} and \ref{hypdz20},
the weighted shift $\slam$ constructed in the proof of
implication (ii)$\Rightarrow$(i) of Theorem
\ref{sqr20} may satisfy either of the following two
conditions:\ $\dz{\slam} \nsubseteq \dz{\slam^*}$ or
$\dz{\slam} \subseteq \dz{\slam^*}$. It turns out that
this general construction always guarantees that
$\dz{\slam^*} \nsubseteq \dz{\slam}$. Indeed, since
for a fixed $u\in V$, $\|\slam e_v\|^2= 1/\lambda_v^2$
for all $v \in \des{u}$ (cf.\ \eqref{kwkw}) and
$\sum_{v \in \dzi{u}} \lambda_v^2 < \infty$, we deduce
that the function $\phi\colon\dzi u \ni v \mapsto
\|\slam e_v\| \in \cbb$ is unbounded, and thus the
operator $M_u$ in $\ell^2(\dzi u)$ of multiplication
by $\phi$ is unbounded (note that the function
$\lambdab^u\colon \dzi u \ni v \mapsto \lambda_v \in
\cbb$ does not belong to $\dz{M_u}$, and so the
definition \cite[(4.2.2)]{j-j-s} makes no sense).
Applying \cite[Theorem 4.2.2]{j-j-s}, we conclude that
$\dz{\slam^*} \nsubseteq \dz{\slam}$.
   \end{rem}
   \begin{rem}
It is worth pointing out that if $\tcal$ is a directed
tree such that $\card{\dzi{u}}=\aleph_0$ for every $u
\in V$, $\slam$ is a densely defined weighted shifts
on $\tcal$ with nonzero weights
$\lambdab=\{\lambda_v\}_{v\in V^\circ}$ such that
$\dz{\slam^2}=\{0\}$ (cf.\ Theorem \ref{sqr20}) and
$v_0 \in V^\circ$, then the weighted shift
$S_{\tilde{\lambdab}}$ on $\tcal$ with nonzero weights
$\tilde{\lambdab}=\{\tilde\lambda_v\}_{v \in V^\circ}$
given by
   \begin{align*}
\tilde\lambda_v =
   \begin{cases}
   \lambda_v & \text{ for } v\neq v_0,
\\
   \sqrt{1 + \|\slam e_{v}\|^2} & \text{ for } v=v_0,
   \end{cases}
   \end{align*}
is densely defined,
$\dz{\slam}=\dz{S_{\tilde{\lambdab}}}$ (use
\cite[Proposition 3.1.3(i)]{j-j-s}),
$\dz{\slam^*}=\dz{S_{\tilde{\lambdab}}^*}$ (use
\cite[Proposition 3.4.1(iv)]{j-j-s}),
$S_{\tilde{\lambdab}}$ is not hyponormal (use Theorem
\ref{hyp}) and $\dz{S_{\tilde{\lambdab}}^2}=\{0\}$
(use \eqref{dof2}). Hence, if $\slam$ is constructed
as in the proof of Theorem \ref{hypdz20}, then by
Remark \ref{dziedz2} we have
$\dz{S_{\tilde{\lambdab}}} \varsubsetneq
\dz{S_{\tilde{\lambdab}}^*}$.
   \end{rem}

\vspace{1ex}

   \textbf{Acknowledgement}. The substantial part of
this paper was written while the first and the third
authors visited Kyungpook National University during
the autumn of 2010 and the spring of 2011. They wish
to thank the faculty and the administration of this
unit for their warm hospitality.
   \bibliographystyle{amsalpha}
   
   \end{document}